\documentclass[12pt]{article}

\usepackage[T1]{fontenc}
\usepackage[utf8]{inputenc}
\usepackage{fullpage}

\usepackage{amssymb, amsmath, amsthm}
\usepackage{graphicx}

\newtheorem{theorem}{Theorem}
\newtheorem*{theoremCL}{Theorem CL (\cite{CL})}
\newtheorem{lemma}[theorem]{Lemma}

\newtheorem{prop}[theorem]{Proposition}
\newtheorem{obs}[theorem]{Observation}
\theoremstyle{remark}
\newtheorem{problem}{Problem}

\theoremstyle{definition}

\title{\hspace*{-8pt} Gallai colorings and domination in multipartite digraphs}

\author{{\bf Andr\'as Gy\'arf\'as}
\thanks{Research partially supported by the Hungarian Foundation for Scientific
Research Grant (OTKA) No.~K68322.}\\
Computer and Automation Research Institute, \vspace*{-1pt} \\
Hungarian Academy of Sciences, \vspace*{-1pt} \\
1518 Budapest, P.O. Box 63, Hungary, \vspace*{-1pt} \\
{\tt gyarfas@sztaki.hu}
\and
{\bf G\'abor Simonyi}
\thanks{Research partially supported by the Hungarian Foundation for
  Scientific Research Grant (OTKA) Nos.~K76088 and NK78439.}\\
Alfr\'ed R\'enyi Institute of Mathematics, \vspace*{-1pt} \\
Hungarian Academy of Sciences, \vspace*{-1pt} \\
1364 Budapest, P.O. Box 127, Hungary, \vspace*{-1pt} \\
{\tt simonyi@renyi.hu}
\and
{\bf \'Agnes T\'oth}
\thanks{Research partially supported by the Hungarian Foundation for Scientific Research Grant and by the National
Office for Research and Technology (Grant number OTKA 67651).} \\
Department of Computer Science and Information Theory, \vspace*{-1pt} \\
Budapest University of Technology and Economics, \vspace*{-1pt} \\
1521 Budapest, P.O. Box 91, Hungary, \vspace*{-1pt} \\
{\tt tothagi@cs.bme.hu}}
\date{}

\begin{document}

\maketitle

\begin{abstract}
Assume that $D$ is a digraph without cyclic triangles and its vertices are partitioned into classes $A_1,\dots,A_t$ of independent vertices. A set $U=\cup_{i\in S} A_i$ is called a dominating set of size $|S|$ if for any vertex $v\in \cup_{i\notin S} A_i$ there is a $w\in U$ such that $(w,v)\in E(D)$. Let $\beta(D)$ be the cardinality of the largest independent set of $D$ whose vertices are from different partite classes of $D$. Our main result says that there exists a $h=h(\beta(D))$ such that $D$ has a dominating set of size at most $h$. This result is applied to settle a problem related to generalized Gallai colorings, edge colorings of graphs without $3$-colored triangles.
\end{abstract}

\section{Introduction}

Investigating comparability graphs Gallai \cite{GA} proved an interesting
theorem about edge-colorings of complete graphs that contain no triangle for which all
three of its edges receive distinct colors. Such colorings turned out to be relevant and
Gallai's theorem proved to be useful also in other contexts, see e.g., \cite{BPV, CE, CEL, FMO, Gur, GYS, KS, KST}.

Honoring the above mentioned work of Gallai an edge-coloring of the complete graph is called
a Gallai coloring if there is no completely multicolored triangle. Recently this notion was extended
to other (not necessarily complete) graphs in \cite{GYS1}.

A basic property of Gallai colored complete graphs is that at least
one of the color classes  spans a connected subgraph on the entire
vertex set. In \cite{GYS1} it was proved that if we color the edges
of a not necessarily complete graph $G$ so that no $3$-colored
triangles appear then there is still a large monochromatic connected
component whose size is proportional to $|V(G)|$ where the
proportion depends on the independence number $\alpha(G)$.

In view of this result it is natural to ask whether one can also
span the whole vertex set with a constant number of connected
monochromatic subgraphs where the constant depends only on
$\alpha(G)$. This question led to a problem about existence of
dominating sets in directed graphs that we believe to be interesting
in itself. In this paper we solve this latter problem thereby giving
an affirmative answer to the previous question.

The paper is organized as follows. In Subsection~\ref{multi} we
describe our digraph problem and state our results on it. The
connection with Gallai colorings will be explained in Subsection~
\ref{secgall}. Section~\ref{proofs} contains the proofs of the
results in Subsection~\ref{multi}. In Section~\ref{concl} we further
elaborate on a question the proofs give rise to.

\subsection{Dominating multipartite digraphs} \label{multi}

We consider multipartite digraphs, i.e., digraphs $D$ whose vertices are partitioned into classes $A_1,\dots,A_t$ of
independent vertices. Suppose that $S\subseteq [t]$. A set $U=\cup_{i\in S} A_i$ is called a dominating set of size
 $|S|$ if for any vertex $v\in \cup_{i\notin S} A_i$ there is a $w\in U$ such that $(w,v)\in E(D)$. The smallest $|S|$
 for which a multipartite digraph $D$ has a dominating set $U=\cup_{i\in S} A_i$ is denoted by $k(D)$. Let $\beta(D)$ be
 the cardinality of the largest independent set of $D$ whose vertices are from different partite classes of $D$.
 (Such independent sets we sometimes refer to as transversal independent sets.)
An important special case is when $|A_i|=1$ for each $i\in [t]$. In this case
$\beta(D)=\alpha(D)$ and $k(D)=\gamma(D)$, the usual domination number of $D$, the smallest
number of vertices in $D$ whose closed outneighborhoods cover $V(D)$. Our main result is the following theorem.

\begin{theorem}\label{main}
For every integer $\beta$ there exists an integer $h=h(\beta)$ such that the following holds.
If $D$ is a multipartite digraph such that $D$ contains no cyclic triangle and $\beta(D)=\beta$, then $k(D)\le h$.
\end{theorem}

Notice that the condition forbidding cyclic triangles in $D$ is
important even when $|A_i|=1$ for all $i$ and $\beta(D)=1$, i.e. for
tournaments. It is well known that $\gamma(D)$ can be arbitrarily
large for tournaments (see, e.g., in \cite{AS}), so $h(1)$
would not exist without excluding cyclic triangles. 

From the proof of Theorem \ref{main} we will get a
factorial upper bound for $k(D)$ from the recurrence formula
$h(\beta)=3\beta+(2\beta +1)h(\beta-1)$. We have relatively small
upper bounds on $k$ only for $\beta=1,2$.

\begin{theorem}\label{main12}
Suppose that $D$ is a multipartite digraph such that $D$ has no cyclic triangle. If $\beta(D)=1$ then $k(D)=1$ and if $\beta(D)=2$ then $k(D)\le 4$.
\end{theorem}

Though the upper bound on $h(\beta)$ obtained from our proof of
Theorem~\ref{main} is much weaker we could not even rule out the
existence of a bound that is linear in $\beta$. We cannot prove a
linear upper bound even in the special case when every partite class
consists of only one vertex. Nevertheless, we treat this case also
separately and provide a slightly better bound than the one
following from Theorem~\ref{main}. The class of digraphs we have here, i.e.,
those with no directed triangles, is studied already and called the
class of \emph{clique-acyclic digraphs}, see \cite{AH}.
\begin{theorem}\label{trivpart}
Let $f(1)=1$ and for $\alpha\ge 2$, $f(\alpha)=\alpha+\alpha f(\alpha-1)$. If $D$ is a clique-acyclic digraph then $\gamma(D)\le f(\alpha(D))$.
\end{theorem}

Apart from the obvious case $\alpha(D)=1$ (when $D$ is a transitive tournament) we know the best possible bound only for $\alpha(D)=2$.

\begin{theorem}\label{triv2}
If $D$ is a clique-acyclic digraph with $\alpha(D)=2$, then $\gamma(D)\le 3$.
\end{theorem}

Note that Theorem~\ref{triv2} is sharp as shown by the cyclically oriented pentagon. Moreover, the union of $t$ vertex disjoint cyclic pentagons shows that we can have $\alpha(D)=2t$ and $\gamma(D)=3t$. Thus in case a linear upper bound would be valid at least in the special case of clique-acyclic digraphs, it could not be smaller than $\frac{3}{2}\alpha(D)$. There are some easy subcases though when the bound is simply $\alpha(D)$.

\begin{prop} \label{perfect}
If $D$ has an acyclic orientation or $D$ is a clique-acyclic perfect graph then $\gamma(D)\le \alpha(D)$.
\end{prop}

Note that Proposition~\ref{perfect} is sharp in the sense that every
graph $G$ has a clique-acyclic orientation resulting in digraph $D$
with $\gamma(D)=\alpha(G)=\alpha(D)$. Indeed, an acyclic orientation
of $G$ where every vertex of a fixed maximum independent set has
indegree zero shows this. It is worth noting the interesting result
of Aharoni and Holzman \cite{AH} stating that a clique-acyclic
digraph always has a fractional kernel, i.e., a fractional
independent set, which is also fractionally dominating.
\medskip

We will see in Section~\ref{proofs} from the proof of
Theorems~\ref{main} and \ref{main12} that the dominating sets we
find there contain two kinds of partite classes. The first kind
could be substituted by just one vertex in it, while the second kind
is chosen not so much to dominate others but because it is itself
not dominated by others. That is, apart from a bounded number of
exceptional partite classes we will dominate the rest of our digraph
with a bounded number of vertices. In Section~\ref{concl} we will
prove another theorem showing that the exceptional classes are
indeed needed.

\subsection{Application to Gallai colorings}\label{secgall}

Recall that Gallai colorings are originally defined as edge-colorings of complete graphs where no triangle gets three different colors. As already mentioned earlier, one of the basic properties of Gallai colorings is that at least one color spans a connected subgraph, i.e. forms a component covering all vertices of the underlying complete graph. In \cite{GYS1} the notion was extended to arbitrary graphs and it was proved that in this setting there is still a large monochromatic connected component. More precisely the following was proved.

\begin{theorem}\label{largecomp} {\rm (\cite{GYS1})}
Suppose that the edges of a graph $G$ are colored so that no triangle is colored with three distinct colors. Then there is a monochromatic component in $G$ with at least ${|V(G)|\over \alpha^2(G)+\alpha(G)-1}$ vertices.
\end{theorem}

Another, in a sense stronger possible generalization of the above basic property of Gallai colorings is also suggested by Theorem \ref{largecomp}, this was asked first at a workshop at Fredericia in November, 2009.

\begin{problem}\label{pro}
Suppose that the edges of a graph $G$ are colored so that no triangle is colored with three distinct colors. Is it true that the vertices of $G$ can be covered by the vertices of at most $k$ monochromatic components where $k$ depends only on $\alpha(G)$?
\end{problem}

\noindent
We remark that an example in \cite{GYS1} shows that even if the $k$ of Problem~\ref{pro} exists, it must be at least ${c\alpha^2(G)\over \log{\alpha(G)}}$ where $c$ is a small constant.

\medskip
\noindent
Theorem~\ref{main} implies an affirmative answer to Problem \ref{pro}. Let $g(1)=1$ and for $\alpha\ge 2$, let $g(\alpha)=g(\alpha-1)+h(\alpha)$ where $h$ is the function given by Theorem \ref{main}.

In the sequel we will use the notation $G[A]$ that denotes the subgraph of graph $G$ induced by $A\subseteq V(G)$.

\begin{theorem}\label{gall}
Suppose that the edges of a graph $G$ are colored so that no triangle is colored with three distinct colors. Then the vertices of $G$ can be covered by the vertices of at most $g(\alpha(G))$ monochromatic components. In case $\alpha(G)=2$ at most five components are enough.
\end{theorem}

\noindent
Note that the last statement of Theorem~\ref{gall} generalizes Theorem \ref{largecomp} in the $\alpha(G)=2$ case.

\medskip
\begin{proof}
For $\alpha(G)=1$ the result is obvious by Gallai's theorem. For $\alpha(G)\ge 2$, suppose that $v\in V(G)$ and let $X$ be the set of vertices in $G$ that are not adjacent to $v$. By induction, the subgraph $G[X]$ can be covered by the vertices of $g(\alpha(G)-1)$ monochromatic components. Let $t$ be the number of colors used on edges of $G$ incident to $v$ and let $A_i$ be the set of vertices incident to $v$ in color $i$. Observe that the condition on the coloring implies that edges of $G$ between $A_i,A_j$ are colored with either color $i$ or color $j$ whenever $1\le i <j\le t$. Thus orienting all edges of color $i$ outward from $A_i$ for every $i$, all edges of $G$ between different classes $A_j$ are  oriented. Moreover, in this orientation there are no cyclic triangles. Thus Theorem \ref{main} is applicable to the oriented subgraph $H$ spanned by the union of the classes $A_j$ after the edges inside the $A_j$'s are removed. We obtain at most $h(\alpha(G))$ dominating sets $A_i$ and each set $v\cup A_i$ together with the vertices that $A_i$ dominates form a connected subgraph of $G$ in color $i$. Thus all vertices of $G$ can be covered by at most $g(\alpha(G)-1)+h(\alpha(G))=g(\alpha(G))$ connected components. In case of $\alpha(G)=2$ we can use Theorem \ref{main12} to get a covering with at most five monochromatic components.
\end{proof}

\section{Proofs}\label{proofs}

We will use the following notation throughout. If $D$ is a digraph and
$U\subseteq V(D)$ is a subset of its vertex set then $N_+(U)=\{v\in V(D):
\exists u\in U\ (u,v)\in E(D)\}$ is the {\em outneighborhood} of $U$. The {\em
  closed outneighborhood} $\hat N_+(U)$ of $U$ is meant to be the set $U\cup N_+(U)$. When $U=\{u\}$ is a single vertex we also write $N_+(u)$ and $\hat N_+(u)$ for $N_+(U)$ and $\hat N_+(U)$, respectively. When $(u,v)\in E(D)$, we will often say that $u$ {\em sends an edge to} $v$.

\medskip
We first deal with the case $\beta(D)=1$ and prove the first statement of Theorem~\ref{main12}. As it will be used several times later, we state it separately as a lemma.

\begin{lemma}\label{compl}
Let $D$ be a multipartite digraph with no cyclic triangle. If $\beta(D)=1$ then $k(D)=1$.
\end{lemma}

\begin{proof}
Let $K$ be a partite class for which $|\hat N_+(K)|$ is largest. We
claim that $K$ is a dominating set. Suppose indirectly, that there
is a vertex $l$ in a partite class $L\neq K$, which is not dominated
by $K$. Since all edges between distinct partite classes are present
in $D$ with some orientation, $l$ must send an edge to all vertices
of $K$. Furthermore, if a vertex $m$ in a partite class $M\neq K,L$
is an outneighbor of some $k\in K$ then it is also an outneighbor of
$l$, otherwise $m$, $l$ and $k$ would form a cyclic triangle. Thus
$\hat N_+(K)\subseteq \hat N_+(L)$. Moreover, $l\in \hat
N_+(L)\setminus \hat N_+(K)$, so $|\hat N_+(L)|>|\hat N_+(K)|$
contradicting the choice of $K$. This completes the proof of the
lemma.

\begin{center}
\includegraphics[scale=0.8]{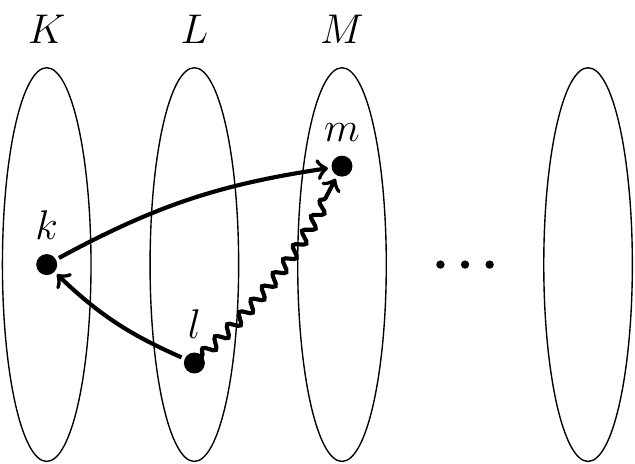}
\end{center}

\end{proof}

In the following two subsections we prove Theorems~\ref{main12} and \ref{main}, respectively.

\subsection{At most $2$ independent vertices}

To prove the second statement of Theorem~\ref{main12} we will need the following stronger variant of Lemma~\ref{compl}.

\begin{lemma}\label{eros}
Let $D$ be a multipartite digraph with no cyclic triangle and $\beta(D)=1$. Then there is a partite class $K$ which is a dominating set, and there is a vertex $k\in K$ such that $V(D)\setminus (K\cup L)\subseteq N_+(k)$ for some partite class $L\neq K$.
\end{lemma}

Thus Lemma~\ref{eros} states that the dominating partite class $K$ has an element that alone dominates almost the whole of $D$, there may be only one exceptional partite class $L$ whose vertices are not dominated by this single element of $K$.

For proving Lemma~\ref{eros}, the following observations will be
used, where $X,Y,Z$ will denote partite classes.

\begin{obs}~\label{obs1}
Let $D$ be a multipartite digraph with no cyclic triangle and $\beta(D)=1$. Suppose that for vertices $x_1, x_2\in X$ and $y\in Y$ the edges $(x_2,y)$ and $(y,x_1)$ are present in $D$. Then for every $z\in Z\neq X,Y$ with $(x_1,z)\in E(D)$ we also have $(x_2,z)\in E(D)$.
\end{obs}

\begin{proof}
Assume indirectly that for some $z\in Z$ the orientation is such that we have $(x_1,z),$ $(z,x_2)\in E(D)$. Then the edge connecting $z$ and $y$ cannot be oriented either way: $(z,y)\in E(D)$ would give a cyclic triangle on vertices $z, y, x_1$, while $(y,z)\in E(D)$ would create one on $y, z, x_2$.
\begin{center}
\includegraphics[scale=0.8]{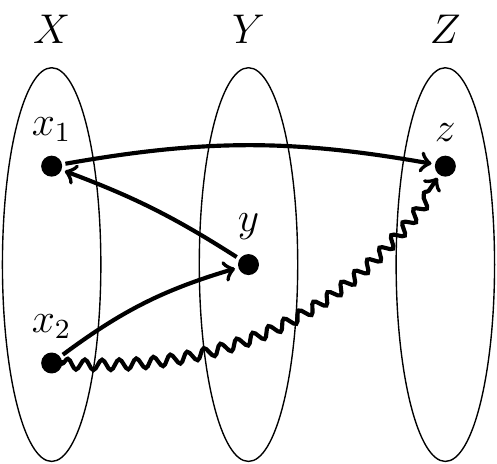}
\end{center}
\end{proof}

\begin{obs}~\label{obs2}
Let $D$ be a multipartite digraph with no cyclic triangle and $\beta(D)=1$. Suppose that for vertices $x_1, x_2\in X$ and $y_1, y_2\in Y$ the edges $(x_1,y_2)$, $(y_2, x_2)$, $(x_2,y_1), (y_1,x_1)$ are present in $D$ forming a cyclic quadrangle. Then in every partite class $Z\neq X,Y$ the outneighborhood of these four vertices is the same.
\end{obs}

\begin{proof}
Let $z$ be an element of $Z\cap N_+(x_1)$. By $(y_1,x_1)\in E(D)$ we must have $z\in Z\cap N_+(y_1)$, otherwise $y_1, x_1, z$ would form a cyclic triangle. Thus we have $Z\cap N_+(x_1)\subseteq Z\cap N_+(y_1)$. Now shifting the role of vertices along the oriented quadrangle backwards we similarly get $Z\cap N_+(x_1)\subseteq Z\cap N_+(y_1)\subseteq Z\cap N_+(x_2)\subseteq Z\cap N_+(y_2)\subseteq Z\cap N_+(x_1)$ proving that we have equality everywhere.
\begin{center}
\includegraphics[scale=0.8]{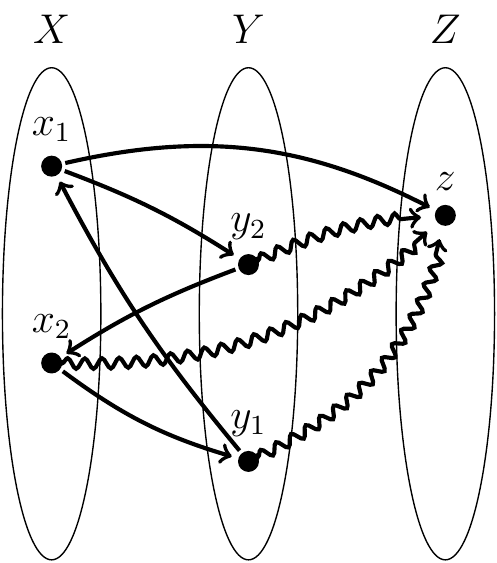}
\end{center}

\end{proof}

Note that in Observation \ref{obs2}, as $\beta (D)=1$, the inneighborhood of the vertices $x_1, x_2, y_1, y_2$ is also the same, so these vertices split to out- and inneighborhood in the same way every partite class $Z\neq X,Y$.

\begin{proof}[Proof of Lemma~\ref{eros}.]

We know from Lemma~\ref{compl} that there is a partite class $K$ which is a dominating set.

Let $k$ be an element of $K$ for which $|N_+(k)|$ is maximal. If $k$ itself dominates all the vertices not in $K$ then we are done. (In that case we do not even need an exceptional class $L$.) Otherwise, there is a vertex $l_1$ in a partite class $L\neq K$ for which the edge between $l_1$ and $k$ is oriented towards $k$. As $L\subseteq N_+(K)$, there must be a vertex $k_1\in K$ which sends an edge to $l_1$.

Using Observation~\ref{obs1} for the vertices $k$, $k_1$ and $l_1$, we obtain that $k_1$ sends an edge not just to $l_1$ but to every vertex in $N_+(k)\setminus L$. By the choice of $k$ this implies the existence of a vertex $l_2\in L$ for which $(k,l_2), (l_2,k_1)\in E(D)$. Thus the vertices $k, l_2, k_1, l_1$ form a cyclic quadrangle. Applying Observation~\ref{obs2} this implies that these four vertices have the same outneighborhood in $V(D)\setminus (K\cup L)$.

We claim that $N_+(k)$ contains all vertices of $D\setminus (K\cup L)$. Assume indirectly, that there is a vertex $m_1$ in a partite class $M\neq K,L$ which is not dominated by $k$. We can argue similarly as we did for $l_1$. Namely, since $M\subseteq N_+(K)$ there is some $k_2\in K$ (perhaps identical to $k_1$) dominating $m_1$. Applying Observation~\ref{obs1} to the vertices $k, m_1$ and $k_2$, we obtain $(N_+(k)\setminus M)\subseteq N_+(k_2)$. Then by the choice of $k$ we must have a vertex $m_2\in M$ for which $(k,m_2), (m_2,k_2)\in E(D)$. So vertices $k, m_2, k_2, m_1$ also form a cyclic quadrangle, and Observation~\ref{obs2} gives us that $Z\cap N_+(k)=Z\cap N_+(m_2)=Z\cap N_+(k_2)=Z\cap N_+(m_1)$ for all partite classes $Z\neq K,M$.

The contradiction will be that the edge between $l_1$ and $m_1$
should be oriented both ways. Indeed, since $(l_1,k)\in E(D)$ and in
$L$ the inneighbors of $k$ and $m_1$ are the same, we must have
$(l_1,m_1)\in E(D)$. However, $(m_1,k)\in E(D)$ and the fact that
$k$ and $l_1$ split $M$ in the same way implies $(m_1,l_1)\in E(D)$.
This contradiction completes the proof of the lemma.

\begin{center}
\includegraphics[scale=0.8]{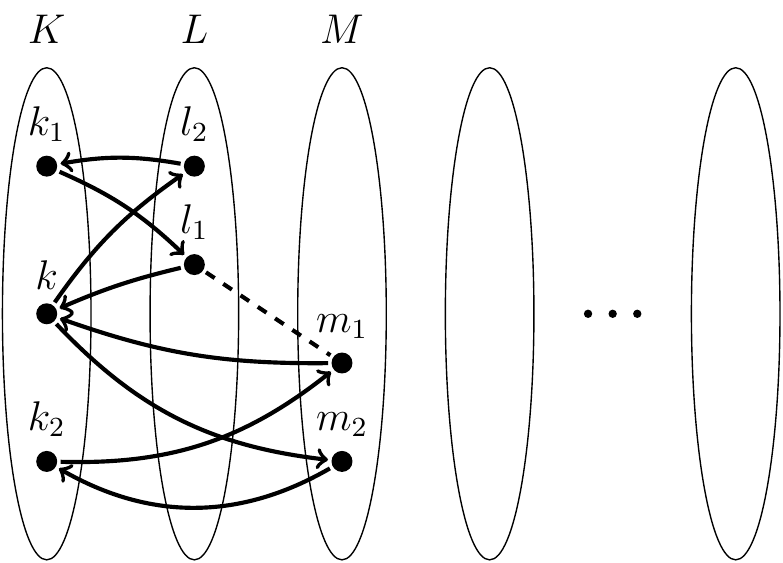}
\end{center}

\end{proof}

Now we are ready to prove the second statement of Theorem~\ref{main12}.

\begin{proof}[Proof of Theorem~\ref{main12}.]
We have already proven the first statement of the theorem. To prove the second part let $D$ be a multipartite digraph with no cyclic triangle and $\beta(D)=2$. We use induction on the number of vertices. The base case is obvious. Let $p$ be a vertex of $D$ and consider the subdigraph $\hat{D}:=D\setminus \{p\}$.

By induction $k(\hat{D})\le 4$. Let $K$, $L$, $M$ and $N$ be four
partite classes of $\hat{D}$ that form a dominating set in $\hat D$.
If $p\in \hat N_+(K\cup L\cup M\cup N)$ then we are done, the same
four sets also dominate $D$. If $p\notin \hat N_+(K\cup L\cup M\cup
N)$ then we will choose four other partite classes that will
dominate $D$. First we choose $P$, the class of $p$. We partition
every other partite class into three parts according to how it is
connected to $p$. For any class $Z$, let $Z_1$ denote the set of
vertices in $Z$ dominated by $p$, let $Z_2$ be the set of vertices
in $Z$ nonadjacent to $p$, and let $Z_3$ denote the set of remaining
vertices of $Z$, i.e., those which send and edge to $p$. We will
refer to $Z_i$ as the $i$-th part of the partite class $Z$, where
$i=1,2,3$. Note that $K_3, L_3, M_3, N_3$ are all empty, otherwise
we would have $p\in \hat N_+(K\cup L\cup M\cup N)$.

Let $D_2$ be the subdigraph of $D$ induced by the vertices in the second part
of the partite classes of $D\setminus P$ in their above partition. $D_2$ is
also a multipartite digraph with no cyclic triangle and $\beta (D_2) =1$.
The latter follows from the fact that the vertices of $D_2$ are all nonadjacent
to $p$ and $\beta(D)=2$. Thus by Lemma~\ref{compl} the vertices of $D_2$ can be
dominated by one partite class $Q_2$, the second part of some partite class $Q$ of $D$.
We choose $Q$ to be the second partite class in our dominating set. Observe that all vertices
of $D$ not dominated so far, i.e., those not in $\hat N_+(P\cup Q)$ should belong to
the third part of their partite classes. Let $u$ be such a vertex.
(If there is none, then we are done.) We know $u\notin K\cup L\cup M\cup N$ as
none of these four classes has a third part. Since $K\cup L\cup M\cup N$ is a
dominating set in $\hat{D}$ there is a vertex $k$ in one of these four classes for
which $(k,u)$ is an edge of $D$. No vertex in the first part of a class can send an
edge to a vertex lying in the third part of some other class, otherwise the latter
two vertices would form a cyclic triangle with $p$. Thus, since $K, L, M, N$
has no third parts, $k$ must be in the second part of one of them.

\begin{center}
\includegraphics[scale=0.8]{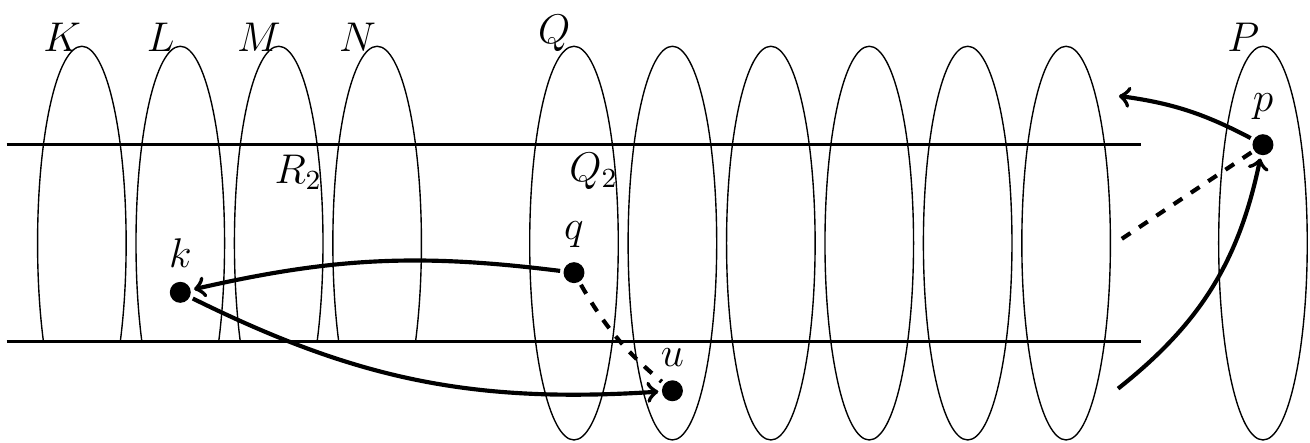}
\end{center}

Lemma~\ref{eros} implies that there is a vertex $q\in Q_2$ with
$V(D_2)\cap\hat N_+(q)$ containing $V(D_2)$ except one exceptional
class $R_2$. We choose $R$, the partite class of $R_2$, to be the
third partite class in our dominating set. ($R$ may or may not be
identical to one of $K$, $L$, $M$, $N$. It is not difficult to see
that we will really need $R$ for the domination only if it is one of
these four classes.) If $u\notin \hat N_+(R)$ then $k$ must be an
outneighbor of $q$. Observe that $(u,q)$ cannot be an edge of $D$,
otherwise $q$, $k$ and $u$ would form a cyclic triangle. But $(q,u)$
cannot be an edge either, as $u\notin N_+(Q)$. Thus $u$ and every so
far undominated vertex is nonadjacent to $q$. Thus the set $U$ of
undominated vertices induces a subgraph $D[U]$ with $\beta
(D[U])=1$, otherwise adding $q$ we would get $\beta(D)\ge 3$. But
then by Lemma~\ref{compl} all vertices in $U$ can be dominated by
only one additional, fourth class.
\end{proof}

\subsection{General case}

Surprisingly, our proof of Theorem~\ref{main} is not a direct generalization of the argument proving
Theorem~\ref{main12} in the previous subsection. In fact, in a way it is conceptually simpler.

\begin{proof}[Proof of Theorem~\ref{main}.]
We have seen that $h(1)=1$ (and $h(2)=4$) is an upper bound for
$k(D)$ if $\beta (D)=1$ (and if $\beta (D)=2$). Now we prove that
$h(\beta)=3\beta +(2\beta +1)h(\beta -1)$ is an upper bound on
$k(D)$ if $\beta (D)=\beta\ge 2$. Let $D$ be a multipartite digraph
with no cyclic triangle and $\beta(D)=\beta$. Let $k_1, k_2, \dots,
k_{2\beta}$ be vertices of $D$, each from a different partite class,
such that $|\hat N_+(\cup_{i=1}^{2\beta}\{k_i\})|$ is maximal. Let
the partite class of $k_i$ be $K_i$ for all $i$ and let ${\cal K}$
denote $\cup_{i=1}^{2\beta} \{k_i\}$. First we declare the $2\beta$
partite classes of these vertices $k_i$ to be part of our dominating
set. Next we partition every other partite class into $2\beta +2$
parts. For an arbitrary partite class $Z\neq K_i$ $(i=1,\dots,
2\beta)$ we denote by $Z_0$ the set $Z\cap N_+({\cal K})$. For
$i=1,2,\dots, 2\beta$ let $Z_i$ be the set of vertices in
$Z\setminus Z_0$ that are not sending an edge to $k_i$, but sending
an edge to $k_j$ for all $j<i$. Finally, we denote by $Z_{2\beta
+1}$, the remaining part of $Z$, that is the set of those vertices
of $Z$ that send an edge to all vertices $k_1, k_2, \dots,
k_{2\beta}$. (As in the proof of Theorem~\ref{main12} we will refer
to the set $Z_i$ as the $i$-th part of $Z$.) The subgraph $D_i$ of
$D$ induced by the $i$-th parts of the partite classes of $D\setminus
(\cup_{i=1}^{2\beta} K_i)$ is also a multipartite digraph with no
cyclic triangle. For $1\le i\le 2\beta$ it satisfies $\beta (D_i)\le
\beta -1$, since adding $k_i$ to any transversal independent set of
$D_i$ we get a larger transversal independent set. So by induction
on $\beta$, each of these $2\beta$ digraphs $D_i$ can be dominated by
at most $h(\beta -1)$ partite classes. We add the appropriate
$2\beta h(\beta -1)$ partite classes to our dominating set.

If $\beta(D_{2\beta +1})\le \beta -1$ also holds then the whole graph can be dominated by choosing additional $h(\beta -1)$ partite classes. Otherwise let $\mathcal{L}=\{l_1,l_2,\dots, l_\beta\}$ be an independent set of size $\beta$ with all its vertices in $V(D_{2\beta +1})$ belonging to distinct partite classes (of $D$), that are denoted by $L_1, L_2, \dots, L_\beta$, respectively. We claim that in the remaining part of $D_{2\beta +1}$, i.e., in $D_{2\beta +1}\setminus (\cup_{i=1}^{\beta}L_i)$ there is no other independent set of size $\beta$ with all elements belonging to different partite classes. Assume indirectly that $m_1\in M_1,m_2\in M_2,\dots, m_\beta\in M_\beta$ form such an independent set $\mathcal{M}$. As $\cal{L}$ is a maximal transversal independent set, every element of a partite class different from $L_1, \dots, L_\beta$ is connected to at least one of the $l_i$'s. And since every element of $\cal{L}$ sends an edge to all the vertices $k_1,\dots, k_{2\beta}$, we must have $N_+({\cal K})\setminus (\cup_{i=1}^{\beta}L_i)\subseteq N_+(\cal L)$ otherwise a cyclic triangle would appear. (The latter is because if $k_i$ ($i\in\{1,2,\dots, 2\beta\}$) sends an edge to $v$, and $l_j$ ($j\in\{1,2,\dots, \beta\}$) sends an edge to $k_i$, moreover $l_j$ is connected with $v$ then the edge between $l_j$ and $v$ must be oriented towards $v$.)

\begin{center}
\includegraphics[scale=0.8]{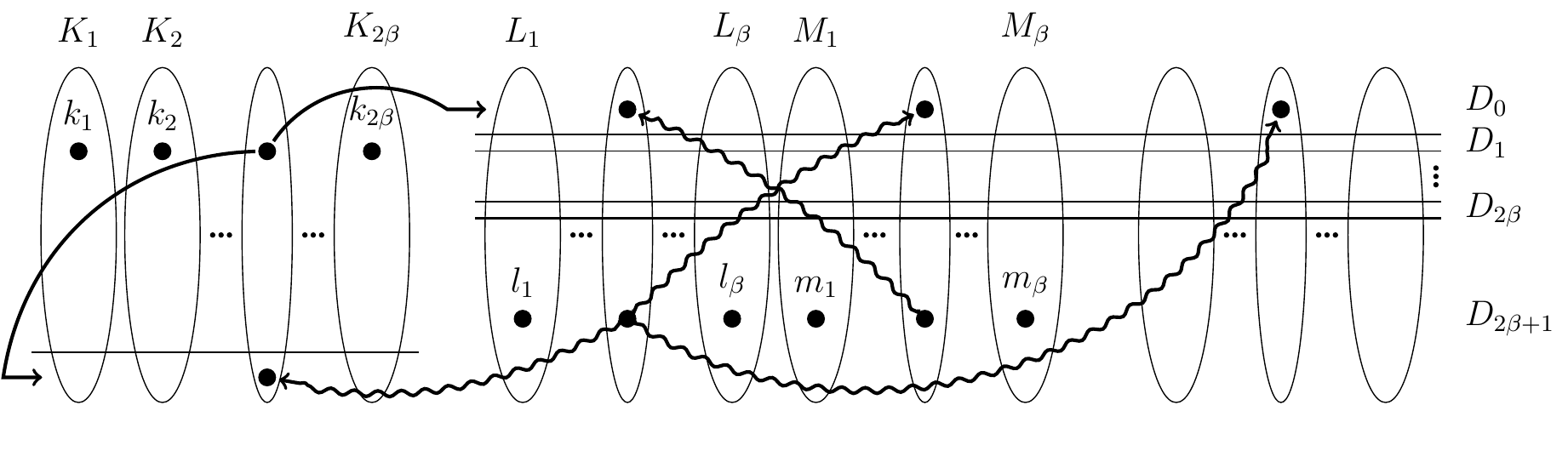}
\end{center}

Similarly, we have $N_+({\cal K})\setminus (\cup_{i=1}^{\beta}M_i)\subseteq N_+(\cal M)$. Thus if such an $\mathcal{M}$ exists then $\hat N_+({\cal K})\subseteq N_+({\cal L}\cup {\cal M})$ while $\hat N_+({\cal L}\cup {\cal M})$ also contains the additional vertices belonging to ${\cal L}\cup {\cal M}$. This contradicts the choice of ${\cal K}$. (Note that ${\cal L}\cup {\cal M}$ dominates also the vertices in $(K_1\cup\dots\cup K_{2\beta})\cap (N_+(k_1)\cup\dots\cup N_+(k_{2\beta}))$.) Thus if we add the classes $L_1, \dots, L_\beta$ to our dominating set, the still not dominated part of $D$ can be dominated by $h(\beta -1)$ further classes. So we constructed a dominating set of $D$ containing at most $2\beta +2\beta h(\beta -1)+\beta +h(\beta -1) = 3\beta +(2\beta +1)h(\beta -1)$ partite classes. This proves the statement.
\end{proof}

Note that we have proved a little bit more than stated in Theorem~\ref{main}. Namely, we showed that there is a set of at most $h_1(\beta)$ vertices of $D$ which dominates the whole graph except perhaps their own partite classes and at most $h_2(\beta)$ other exceptional classes. From the proof we obtain the recursion formula $h_1(\beta)\le 2\beta+(2\beta+1) h_1(\beta-1)$ and $h_2(\beta)\le \beta+(2\beta+1) h_2(\beta -1)$.

\subsection{Clique-acyclic digraphs}

For the proof of Theorem~\ref{trivpart} we will use the following theorem due to Chv\'atal and Lov\'asz \cite{CL}.

\medskip\noindent
\begin{theoremCL}
Every directed graph $D$ contains a semi-kernel, that is an
independent set $U$ satisfying that for every vertex $v\in D$ there
is an $u\in U$ such that one can reach $v$ from $u$ via a directed
path of at most two edges.
\end{theoremCL}

\begin{proof}[Proof of Theorem~\ref{trivpart}.]
The statement is trivial for $\alpha(D)=1$, since a transitive tournament is dominated by its unique vertex of indegree $0$. We use induction on $\alpha=\alpha(D)$. Assume the theorem is already proven for $\alpha-1$. Consider $D$ with $\alpha(D)=\alpha$ and a semi-kernel $U$ in $D$ that exists by Theorem CL. We define a set $S$ with $|S|\le f(\alpha)$ elements dominating each vertex. Let $U\subseteq S$. Then $S$ already dominates the neighborhood of $U$. Denote by $T$ the second outneighborhood of $U$ (i.e., the set of all vertices not in $U$ and not yet dominated). Observe that for every vertex $w\in T$ there is a vertex $u\in U$ such that neither $(u,w)$ nor $(w,u)$ is an edge. Indeed, let $u$ be the vertex of $U$ from which $w$ can be reached by traversing two directed edges. Then $(w,u)\notin E(D)$ otherwise we would have a cyclic triangle. But $(u,w)\notin E(D)$ is immediate from knowing that $w$ is not in the first outneighborhood of $U$.
\smallskip\noindent
Partition $T$ into $|U|\le\alpha$ classes $L_u$ indexed by the elements of $U$ where $w\in L_u$ means that $u$ and $w$ are nonadjacent. Thus all vertices in each class $L_u$ are independent from the same vertex in $U$ implying that the induced subgraph $D[L_u]$ has independence number at most $\alpha-1$. Thus $D[L_u]$ can be dominated by at most $f(\alpha-1)$ vertices. Add these to $S$ for every $u\in U$. So all vertices can be dominated by at most $\alpha+\alpha f(\alpha-1)= f(\alpha)$ vertices completing the proof.

\begin{center}
\includegraphics[scale=0.8]{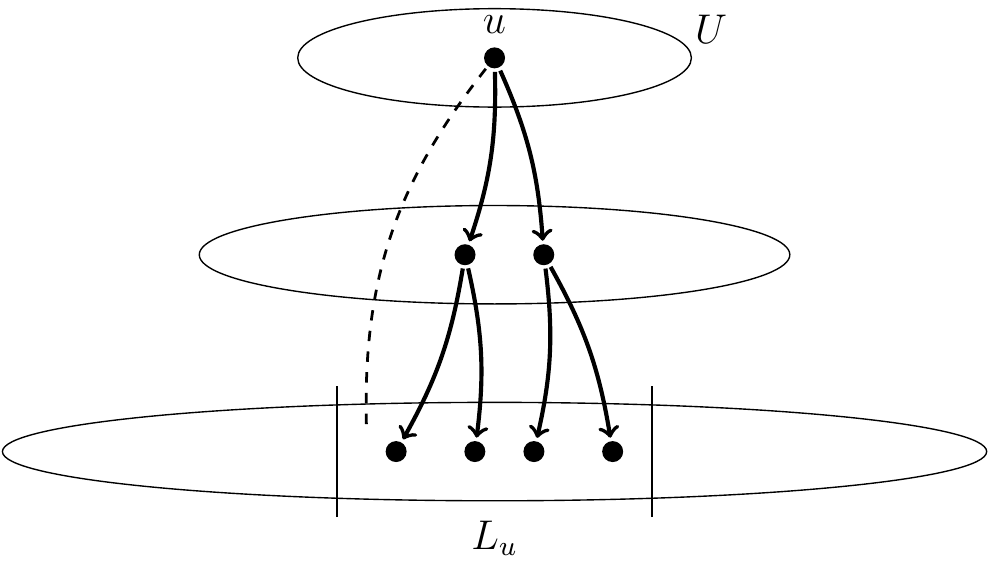}
\end{center}

\end{proof}

For $\alpha(D)=2$ the above theorem gives $\gamma(D)\le f(2)=4$. Compared to this the improvement of Theorem~\ref{triv2} is only $1$, but as already mentioned, the cyclically oriented five-cycle shows that $\gamma(D)\le 3$ is the best possible upper bound.

The proof of Theorem~\ref{triv2} goes along similar lines as the proof we had for the second statement of Theorem~\ref{main12}.

\begin{proof}[Proof of Theorem~\ref{triv2}.]
We use induction on the number of vertices in $D$. Let $p$ be a vertex of $D$, and partition the remaining vertices of $D$ into three parts. Let $V_1$ be the set of vertices that are dominated by $p$, $V_2$ the set of vertices nonadjacent to $p$, and let $V_3$ be the set of vertices which send an edge to $p$. We assume by induction that $D\setminus\{p\}$ can be dominated by three vertices. (The base case is obvious.) If at least one of these is located in $V_3$ then $p$ is also dominated by them and we are done. Otherwise we create a new dominating set. First we choose $p$, and by $p$ we dominate all the vertices in $V_1$. Observe that any two vertices in $V_2$ must be connected, because two nonadjacent vertices of $V_2$ and $p$ would form an independent set of size $3$. Thus $D[V_2]$ is a transitive tournament and so it can be dominated by just one vertex, let it be $q\in V_2$. Let $U$ be the set of remaining undominated vertices. That is, $U=V_3\setminus N_+(q)$. Consider an arbitrary element $u\in U$. We know that $u$ is dominated by a vertex of the dominating set of $D\setminus\{p\}$. Let this vertex be $k$, clearly it cannot belong to $V_3$ (then it would dominate $p$). We also have $k\notin V_1$, otherwise there is a cyclic triangle on the vertices $p,k$, and $u$. So $k\in V_2$, and thus $q$ sends an edge to $k$. Since $u$ is undominated, $(q,u)$ is not an edge of $D$. With the edge $(u,q)$, we would get a cyclic triangle on $u$, $q$ and $k$. So $u$ and all the vertices in $U$ are nonadjacent to $q$, therefore $\alpha(D[U])=1$ and thus $U$ can be dominated by one vertex $r$. Thus all vertices of $D$ are dominated by the $3$-element set $\{p, q, r\}$. This completes the proof.

\begin{center}
\includegraphics[scale=0.8]{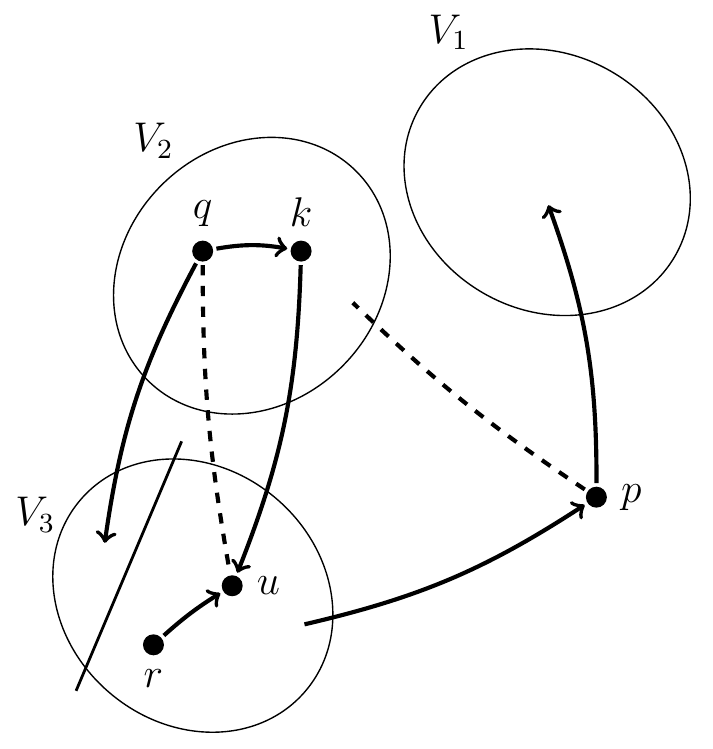}
\end{center}

\end{proof}

\bigskip\noindent
To prove Proposition~\ref{perfect} we formulate the following simple observation. Let $\chi(F)$ denote the chromatic number of graph $F$.

\begin{obs}\label{chi}
Let $D$ be a directed graph and $\bar D$ the complementary graph of the undirected graph underlying $D$. If $D$ is clique-acyclic, then $\gamma(D)\le\chi(\bar D)$.
\end{obs}

\begin{proof}
It follows from the definition of $\chi(\bar D)$ that the vertex set of $D$ can be covered by $\chi(\bar D)$ complete subgraphs of $D$. Since $D$ is clique-acyclic, all these complete subgraphs can be dominated by one of their vertices. Thus all vertices are dominated by these $\chi(\bar D)$ chosen vertices.
\end{proof}

\begin{proof}[Proof of Proposition~\ref{perfect}.]
If the orientation of $D$ is acyclic, then consider those vertices
that have indegree zero. Let these form the set $U_0$. Delete these
vertices and all vertices they dominate. Let set $U_1$ contain the
indegree zero vertices of the remaining graph, and delete the
vertices in $U_1\cup N_+(U_1)$. Proceed this way to form the sets
$U_2, \dots, U_s$, where finally there are no remaining vertices
after $U_s$ and its neighbors are deleted. It follows from the
construction that $U_0\cup U_1\cup\dots \cup U_s$ is an independent
set and dominates all vertices not contained in it.

The second statement immediately follows from Observation~\ref{chi} and the fact that $\chi(\bar D)=\alpha(D)$ if $D$ is perfect, an immediate consequence of the Perfect Graph Theorem \cite{LLperf}.
\end{proof}

\section{On the exceptional classes}\label{concl}
As already mentioned in the Introduction and also after the proof of
Theorem~\ref{main}, the statement of Theorem~\ref{main} could be
formulated in a somewhat stronger form. Namely, we do not only
dominate our multipartite digraph $D$ by $h(\beta)$ partite classes,
we actually dominate almost all of $D$ by $h_1(\beta)$ vertices,
where ``almost'' means that there is only a bounded number
$h_2(\beta)$ of partite classes not dominated this way. The first
appearance of this phenomenon is in Lemma~\ref{eros} where we showed
that if $\beta(D)=1$ then a single vertex dominates the whole graph
except at most one class. To complement this statement we show below
that this exceptional class is indeed needed, we cannot expect to
dominate the whole graph by a constant number of  vertices. In other
words, if we want to dominate with a constant number of singletons
(and not by simply taking a vertex from each partite class), then we
do need exceptional classes already in the $\beta(D)=1$ case.

For a bipartite digraph $D$ with partite classes $A$ and $B$ let
$\gamma_A(D)$ denote the minimum number of vertices in $A$ that
dominate $B$ and similarly let $\gamma_B(D)$ denote the minimum
number of vertices in $B$ dominating $A$. Let
$\gamma_0(D)=\min\{\gamma_A(D),\gamma_B(D)\}$.

\begin{theorem}
There exists a sequence of oriented complete bipartite graphs $\{D_k\}_{k=1}^{\infty}$ satisfying $\gamma_0(D_k)>k$.
\end{theorem}

We note that the existence of $D_k$ with $n$ vertices in each
partite class and satisfying $\gamma_0(D_k)>k$ follows by a standard
probabilistic argument provided that \mbox{$2{n\choose k}(1-2^{-k})^n<1$}. Our
proof below is constructive, however. 

\begin{proof}
We give a simple recursive construction for $D_k$ in which we blow up the
vertices of a cyclically oriented cycle $C_{2k+2}$ and connect the blown up
versions of originally nonadjacent vertices that are an odd distance away
from each other by copies of the already constructed digraph $D_{k-1}$. 

Let $D_1$ be a cyclic $4$-cycle, i.e., a cyclically oriented
$K_{2,2}$. It is clear that neither partite class in this digraph
can be dominated by a single element of the other partite class.
Thus $\gamma_0(D_1)>1$ holds.

Assume we have already constructed $D_{k-1}$ satisfying $\gamma_0(D_{k-1})>k-1$. Let the two partite classes of $D_{k-1}$ be $A_{k-1}=\{a_1, \dots, a_m\}$ and $B_{k-1}=\{b_1, \dots, b_m\}$. Now we construct $D_k$ as follows. Let the vertex set of $D_k$ be $V(D_k)=A_k\cup B_k$, where
$$A_k:=\{(j,a_i): 1\le j\le k+1, 1\le i\le m\},$$
$$B_k:=\{(j,b_i): 1\le j\le k+1, 1\le i\le m\}.$$
There will be an oriented edge from vertex $(j,a_i)$ to $(r,b_s)$ if either $j=r$, or $j\not\equiv r+1 \pmod{k+1}$ and $(a_i,b_s)\in E(D_{k-1})$. All other edges between $A_k$ and $B_k$ are oriented towards $A_k$, i.e., this latter set of edges can be described as
$$\{((r,b_s),(j,a_i))\, :\, j\equiv r+1\!\!\!\! \pmod{k+1}\quad {\rm or}\quad ((b_s,a_i)\in E(D_{k-1})\ {\rm and}\ j\ne r)\}.$$

\smallskip

It is only left to prove that $\gamma_0(D_k)>k$. Let us use the notation
$A_k(j)=\{(j,a_i): 1\le i\le m\}$, $B_k(j)=\{(j,b_i): 1\le i\le m\}$. Consider
a set $K$ of $k$ vertices of $A_k$, we show it cannot dominate $B_k$. There
must be an $r\in\{1,\dots, k+1\}$ by pigeon-hole for which $K\cap
A_k(r)=\emptyset$ and $K\cap A_k(r+1)\neq\emptyset$. (Addition here is meant
modulo $(k+1)$.) Fix this $r$. We claim that some vertex in $B_k(r)$ will not
be dominated by $K$. Indeed, the vertex $(r+1,a_i)\in K\cap A_k(r+1)$ does not
send any edge into $B_k(r)$, so we have only at most $k-1$ vertices in $K$ that can dominate vertices in $B_k(r)$ and all these vertices are in $A_k\setminus A_k(r)$. Notice that the induced subgraph of $D_k$ on $B_k(r)\cup A_k\setminus A_k(r)$ admits a digraph homomorphism (that is an edge-preserving map) into $D_{k-1}$. Indeed, the projection of each vertex to its second coordinate gives such a map by the definition of $D_k$. So if the above mentioned $k-1$ vertices would dominate the entire set $B_k(r)$, then their homomorphic images would dominate the homomorphic image of $B_k(r)$ in $D_{k-1}$. The latter image is the entire set $B_{k-1}$ and by our induction hypthesis it cannot be dominated by $k-1$ vertices of $A_{k-1}$. Thus we indeed have $\gamma_{A_k}(D_k)>k$.

The proof of $\gamma_{B_k}(D_k)>k$ is similar by symmetry. Thus we have $\gamma_0(D_k)>k$ as stated.
\end{proof}

\end{document}